\documentclass[a4paper,11pt,oneside]{article}
\usepackage{graphicx}
\usepackage{amscd}
\usepackage{amsmath}
\usepackage{caption}
\usepackage{amsfonts}
\usepackage{amssymb}
\usepackage{amsthm}
\usepackage{mathrsfs}
\usepackage{multicol}
\usepackage{color}
\usepackage[english]{babel}
\usepackage[T1]{fontenc}
\usepackage{textcomp}
\usepackage[utf8]{inputenc}
\usepackage{enumitem}
\usepackage{indentfirst}
\usepackage{fancyhdr}

\newcommand{\N}{\mathbb N}

\newcommand{\RR}{{{\rm I} \kern -.15em {\rm R} }}

\begin{document}
	\theoremstyle{plain} \newtheorem{thm}{Theorem}[section] \newtheorem{cor}[thm]{Corollary} \newtheorem{lem}[thm]{Lemma} \newtheorem{prop}[thm]{Proposition} \theoremstyle{definition} \newtheorem{defn}{Definition}[section] 
	
	\newtheorem{oss}[thm]{Remark}
	\newtheorem{ex}{Example}[section]
	\newtheorem{lemma}{Lemma}[section]
	\title{Asymptotic flocking for the Cucker-Smale model with\\ time
		variable time delays}
	\author{{\sc Elisa Continelli \footnote{Dipartimento di Ingegneria e Scienze dell'Informazione e Matematica, Universit\`{a} degli Studi di L'Aquila, Via Vetoio, Loc. Coppito, 67010 L'Aquila, Italy (elisa.continelli@graduate.univaq.it). }}}	
	\maketitle
	\pagestyle{fancy}
	\lhead{\scriptsize Asymptotic flocking for the Cucker-Smale model with time
		variable time delays}
\lhead{ Asymptotic flocking for the Cucker-Smale model with time
	variable time delays}	
\rhead{}
	\begin{abstract}
		In this paper, we investigate a Cucker-Smale flocking model with varying time delay. We establish exponential asymptotic flocking without requiring smallness assumptions on the time delay size and the monotonicity of the influence function.
	\end{abstract}
	
	\noindent\providecommand{\mathsub}[1]{\textbf{2020 Mathematics Subject Classification:} #1}
	\mathsub{ 34D05, 34K20, 92D25, 92D50}	
	\\\providecommand{\keywords}[1]{\textbf{Keywords:} #1}
	
	\vspace{5 mm}
	
	\section{Introduction}
	\setcounter{equation}{0}
	Multiagent systems have been largely studied in these last years, by virtue of their wide application to many scientific fields, such as biology \cite{Cama, CS1}, economics \cite{Marsan}, robotics \cite{Bullo, Jad}, control theory \cite{CFPT, Borzi, PRT, WCB, Aydogdu,  Piccoli}, social sciences \cite{Bellomo, Campi}. In this regard, we mention the Hegselmann-Krause opinion formation model \cite{HK} and the Cucker-Smale flocking model \cite{CS1}. Typically, for the solutions of such systems, the convergence to consensus, in the case of the Hegselmann-Krause model, or the exhibition of asymptotic flocking, in the case of the Cucker-Smale model, is investigated. Also, we mention \cite{HT, carr, CCR} for the analysis of the kinetic version of the Cucker-Smale model. The Cucker-Smale model is presented for nonsymmetric interaction coefficients in \cite{MT}. 
	
	In the theory of multiagent systems, the introduction of time delays is reasonable. For instance, in opinion formation or flocking models, a particle does not necessarily receive information from the other agents instantaneously. In the applications, it is more realistic to contemplate models in which each particle has to wait for a while before collecting all the information coming from the other particles of the system (see \cite{Kuang, Nic}).  
	
	In this paper, we focus on a flocking model with time variable time delays. Multiagent systems with delay have already been studied in some works, among them \cite{LW, DH, H3}. Some flocking results for the delayed Cucker-Smale model have been established in \cite{HM, CH, CL, PT, CP}. In all these works, the authors require a smallness assumption on the time delay size. Also, convergence to consensus for the Hegselmann-Krause model with small time delays has been investigated in \cite{CPP, H, P}. Very recently, a consensus result for the constant time delay case is proved in \cite{H3}, without requiring upper bounds on the time delay. We also refer to \cite{Lu} for a consensus result without any restrictions on the constant time delay but in the particular case of constant interaction coefficients. For the Cucker-Smale model with constant time delay, asymptotic flocking is achieved by \cite{Cartabia} without assuming the smallness of the time delay size. A flocking result for the Cucker-Smale model with leadership and time delay without upper bounds is obtained by \cite{PR}. 
	
	In this paper, we extend the argument of \cite{Cartabia} to a Cucker-Smale flocking model with time-varying time delay. We succed in improving previous flocking results by removing upper bounds on the time delay size. To be precise, we deal with a time delay function $\tau(\cdot):[0,+\infty)\rightarrow[0,+\infty)$ that satisfies
	$$0\leq \tau(t)\leq \bar{\tau},\quad\forall t \geq 0,$$
	for some positive constant $\bar{\tau}$, and we prove the asymptotic flocking for the solutions of the delayed Cucker-Smale model without requiring any smallness assumptions on the time delay size $\bar{\tau}$. Moreover, we are able to weaken the monotonicity hypothesis on the influence function that is assumed in \cite{Cartabia}. Indeed, in \cite{Cartabia} the influence function is a continuous function of the distance between the agents' positions that is required to be nonincreasing. Here, we rather consider an influence function that it is assumed to be just positive, bounded and continuous. 
	\\\\Consider a finite set of $N\in\N$ particles, with $N\geq 2 $. Let $(x_{i}(t))\in \RR^d$ and $(v_{i}(t))\in \RR^d$ denote the position and the velocity of the $i$-th particle at time $t$, respectively. We shall denote with $\lvert\cdot \rvert$ and $\langle \cdot,\cdot \rangle$ the usual norm and scalar product in $\RR^{d}$, respectively. The interactions between the elements of the system are described by the following Cucker-Smale type model with a variable time delay
	\begin{equation}\label{csp}
		\begin{cases}
		\frac{d}{dt}x_{i}(t)=v_{i}(t),\quad &t>0, \,\,\forall i=1,\dots,N,\\\frac{d}{dt}v_{i}(t)=\underset{j:j\neq i}{\sum}a_{ij}(t)(v_{j}(t-\tau(t))-v_{i}(t)),\quad 	&t>0,\,\,\forall i=1,\dots,N,	
		\end{cases}
	\end{equation}
	with weights $a_{ij}$ of the form
	\begin{equation}\label{weight}
	a_{ij}(t):=\frac{1}{N-1}\psi( \lvert x_{i}(t)-x_{j}(t-\tau(t))\rvert), \quad\forall t>0,\, \forall i,j=1,\dots,N,
	\end{equation}
	where $\psi:\RR\rightarrow \RR$ is a positive function. The time delay function $\tau:[0,\infty)\rightarrow[0,\infty)$ is assumed to be continuous and it satisfies  
	\begin{equation}\label{delay}
		0\leq \tau(t)\leq \bar{\tau}, \quad \forall t\geq 0,
	\end{equation}
	for some positive constant $\bar{\tau}$. \\The initial conditions
	\begin{equation}\label{incond}
		x_{i}(s)=x^{0}_{i}(s),\quad v_{i}(s)=v^{0}_{i}(s), \quad \forall s\in [-\bar{\tau},0],\,\forall i=1,\dots,N,
	\end{equation}
	are assumed to be continuous functions.
	\\The influence function $\psi$ is assumed to be continuous. Moreover, we assume that it is bounded and we denote with
	$$K:=\lVert \psi\rVert_{\infty}.$$ 
	For existence results related to system \eqref{csp} we refer to \cite{Halanay, HL}. Here, we are interested in the analysis of the asymptotic flocking of the solutions to \eqref{csp}.
	\\For each $t\geq-\bar{\tau}$, we define the diameters in space and in velocity $$d_{X}(t):=\max_{i,j=1,\dots,N}\lvert x_{i}(t)-x_{j}(t)\rvert,$$
	$$d_{V}(t):=\max_{i,j=1,\dots,N}\lvert v_{i}(t)-v_{j}(t)\rvert.$$
	\begin{defn} (Unconditional flocking)\label{unflock} We say that a solution $\{(x_{i},v_{i})\}_{i=1,\dots,N}$ to system \eqref{csp} exhibits \textit{asymptotic flocking} if it satisfies the two following conditions:
		\begin{enumerate}
			\item there exists a positive constant $d^{*}$ such that$$\sup_{t\geq-\bar{\tau}}d_{X}(t)\leq d^{*};$$
			\item$\underset{t \to\infty}{\lim}d_{V}(t)=0.$
		\end{enumerate}
	\end{defn}
	Our main result is the following.
\begin{thm} (Unconditional flocking)\label{uf}
	Assume that $\psi:\RR\rightarrow\RR$ is a positive, bounded, continuous function that satisfies
	\begin{equation}\label{infint}
	\int_{0}^{+\infty}\min_{r\in [0,x]}\psi(r)dx=+\infty.
	\end{equation}Moreover, let $x^{0}_{i},v_{i}^{0}:[-\bar{\tau},0]\rightarrow \RR^{d}$ be continuous functions, for any $i=1,\dots,N$. 
	Then, for every solution $\{(x_{i},v_{i})\}_{i=1,\dots,N}$ to \eqref{csp} with the initial conditions \eqref{incond}, there exists a positive constant $d^{*}$ such that \begin{equation}\label{posbound}
		\sup_{t\geq-\bar{\tau}}d_{X}(t)\leq d^{*},
	\end{equation}
	and there exists another positive constant $C$, independent of $N$, for which the following exponential decay estimate holds
	\begin{equation}\label{vel}
		d_{V}(t)\leq \left(\max_{i,j=1,\dots,N}\,\,\max_{r,s\in [-\bar{\tau},0]}\lvert v_{i}(r)-v_{j}(s)\rvert\right) e^{-C(t-2\bar{\tau})},\quad \forall t\geq-\bar{\tau}.
	\end{equation}
\end{thm}
\begin{oss}
	Let us note that, if the influence function $\psi$ is nonincreasing, then the assumption \eqref{infint} reduces to 
	\begin{equation}\label{intcart}
		\int_{0}^{+\infty}\psi(x)dx=+\infty.
	\end{equation}
	The condition \eqref{intcart} is the one assumed in \cite{Cartabia} in order to achieve the unconditional flocking for the solutions of the Cucker-Smale model.
\end{oss}
\vspace{0.7cm}The rest of this paper is organized as follows. In Section 2 we present some preliminary definitions and results that will be needed for the proof of Theorem \ref{unflock}. In Section 3 we prove our asymptotic flocking result by introducing a suitable Lyapunov functional. 
\section{Preliminaries}\setcounter{equation}{0}
\noindent Let $\{(x_{i},v_{i})\}_{i=1,\dots,N}$  be solution to \eqref{csp} under the initial conditions \eqref{incond}. In this section, we assume that the hypotheses of Theorem \ref{uf} are satisfied.
\par We now present some auxiliary lemmas that generalize and extend the analogous results in \cite{Cartabia}.
\begin{lem}\label{lemma1}
	For each $v\in \RR^{d}$ and $T\geq 0$, we have that
	\begin{equation}\label{inpro}
		\min_{j=1,\dots,N}\min_{s\in [T-\bar{\tau},T]}\langle v_{j}(s),v\rangle\leq \langle v_{i}(t),v\rangle\leq \max_{j=1,\dots,N}\max_{s\in [T-\bar{\tau},T]}\langle v_{j}(s),v\rangle,
	\end{equation}
	for all $t\geq T-\bar{\tau}$ and $i=1,\dots,N$.
\end{lem}
\begin{proof}
	First of all, we can note that the inequalities in \eqref{inpro} are satisfied for every $t\in [T-\bar{\tau},T]$.
	\\Now, let $T\geq 0$. Given a vector $v\in \RR^{d}$, we set $$M=\max_{j=1,\dots,N}\max_{s\in[T-\bar{\tau},T]}\langle v_{j}(s),v\rangle.$$
	For all $\epsilon >0$, let us define
	$$K^{\epsilon}:=\left\{t>T :\max_{i=1,\dots,N}\langle v_{i}(s),v\rangle < M+\epsilon,\,\forall s\in [T,t)\right\}.$$
	By continuity, we have that $K^{\epsilon}\neq\emptyset$. Thus, denoting with $$S^{\epsilon}:=\sup K^{\epsilon},$$
	it holds that $S^{\epsilon}>T$. \\We claim that $S^{\epsilon}=+\infty$. Indeed, suppose by contradiction that $S^{\epsilon}<+\infty$. Note that, by definition of $S^{\epsilon}$, it turns out that \begin{equation}\label{maximum}
	\max_{i=1,\dots,N}\langle v_{i}(t),v\rangle<M
	+\epsilon,\quad \forall t\in (T,S^{\epsilon}),
	\end{equation}
	and \begin{equation}\label{teps}
	\lim_{t\to S^{\epsilon-}}\max_{i=1,\dots,N}\langle v_{i}(t),v\rangle=M+\epsilon.
	\end{equation}
	For all $i=1,\dots,N$ and $t\in (T,S^{\epsilon})$, we compute
	$$\frac{d}{dt}\langle v_{i}(t),v\rangle=\frac{1}{N-1}\sum_{j:j\neq i}\psi(\lvert x_{i}(t)- x_{j}(t-\tau(t))\rvert)\langle v_{j}(t-\tau(t))-v_{i}(t),v\rangle.$$
	Notice that, being $t\in (T,S^{\epsilon})$, then  $t-\tau(t)\in (T-\bar{\tau}, S^{\epsilon})$ and \begin{equation}\label{t-tau}
	\langle v_{j}(t-\tau(t)),v\rangle< M+\epsilon,\quad \forall j=1, \dots, N.
	\end{equation}
	Moreover, \eqref{maximum} implies that $$\langle v_{i}(t),v\rangle<M+\epsilon,$$
	so that $$M+\epsilon-\langle v_{i}(t),v\rangle\geq 0.$$Combining this last fact with \eqref{t-tau} and by recalling that $\psi$ is bounded, we can write $$\frac{d}{dt}\langle v_{i}(t),v\rangle\leq \frac{1}{N-1}\sum_{j:j\neq i}\psi(\lvert x_{i}(t)- x_{j}(t-\tau(t))\rvert)(M+\epsilon-\langle v_{i}(t),v\rangle)$$$$\leq K(M+\epsilon-\langle v_{i}(t),v\rangle), \quad\forall t\in (T, S^{\epsilon}).$$
	Then, from the Gronwall's inequality we get
	$$\begin{array}{l}
	\vspace{0.3cm}\displaystyle{
		\langle v_{i}(t),v\rangle\leq e^{-K(t-T)}\langle v_{i}(T),v\rangle+K(M+\epsilon)\int_{T}^{t}e^{-K(t-s)}ds}\\
	\vspace{0.3cm}\displaystyle{\hspace{1.5 cm}
		=e^{-K(t-T)}\langle v_{i}(T),v\rangle+(M+\epsilon)(1-e^{-K(t-T)})}\\
	\vspace{0.3cm}\displaystyle{\hspace{1.5 cm}
		\leq e^{-K(t-T)}M+M+\epsilon -Me^{-K(t-T)}-\epsilon e^{-K(t-T)}}\\
	\vspace{0.3cm}\displaystyle{\hspace{1.5 cm}
		=M+\epsilon-\epsilon e^{-K(t-T)}}\\
	\displaystyle{\hspace{1.5 cm}\leq M+\epsilon-\epsilon e^{-K(S^{\epsilon}-T)},}
	\end{array}
	$$for all $t\in (T, S^{\epsilon})$.	We have so proved that, $\forall i=1,\dots, N,$
	$$\langle v_{i}(t),v\rangle\leq M+\epsilon-\epsilon e^{-K(S^{\epsilon}-T)}, \quad \forall t\in (T,S^{\epsilon}).$$
	Thus, we get
	\begin{equation}\label{limit}
	\max_{i=1,\dots,N} \langle v_{i}(t),v\rangle\leq M+\epsilon-\epsilon e^{-K(S^{\epsilon}-T)}, \quad \forall t\in (T,S^{\epsilon}).
	\end{equation}
	Letting $t\to S^{\epsilon-}$ in \eqref{limit}, from \eqref{teps} we have that $$M+\epsilon\leq M+\epsilon-\epsilon e^{-K(S^{\epsilon}-T)}<M+\epsilon,$$
	which is a contraddiction. Thus, $S^{\epsilon}=+\infty$, which means that $$\max_{i=1,\dots,N}\langle v_{i}(t),v\rangle<M+\epsilon, \quad \forall t>T.$$
	From the arbitrariness of $\epsilon$ we can conclude that $$\max_{i=1,\dots,N}\langle v_{i}(t),v\rangle\leq M, \quad \forall t>T,$$
	from which $$\langle v_{i}(t),v\rangle\leq M, \quad \forall t>T, \,\forall i=1,\dots,N,$$
	which proves the second inequality in \eqref{inpro}. 
	\\Now, to prove the other inequality, let $v\in \RR^{d}$ and define $$m=\min_{j=1,\dots,N}\min_{s\in[T-\bar{\tau},T]}\langle v_{j}(s),v\rangle.$$
	Then, for all $i=1,\dots,N$ and $t>T$, by applying the second inequality in \eqref{inpro} to the vector $-v\in\RR^{d}$ we get $$-\langle v_{i}(s),v\rangle=\langle v_{i}(t),-v\rangle\leq \max_{j=1,\dots,N}\max_{s\in[T-\bar{\tau},T]}\langle v_{j}(s),-v\rangle$$$$\quad\quad\,\,\,=-\min_{j=1,\dots,N}\min_{s\in[T-\bar{\tau},T]}\langle v_{j}(s),v\rangle=-m,$$
	from which $$\langle v_{i}(s),v\rangle\geq m.$$
	Thus, also the first inequality in \eqref{inpro} is fullfilled.
\end{proof}
We now introduce some notation.
\begin{defn}\label{notation}
	We define 
	$$D_{0}=\max_{i,j=1,\dots,N}\,\,\max_{s,t\in [-\bar{\tau},0]}\lvert v_{i}(s)-v_{j}(t)\rvert,$$
	and in general, $\forall n\in \mathbb{N}$,
	$$D_{n}:=\max_{i,j=1,\dots,N}\,\,\max_{s,t\in [n\bar{\tau}-\bar{\tau},n\bar{\tau}]}\lvert v_{i}(s)-v_{j}(t)\rvert.$$
\end{defn}
\noindent Notice that inequality \eqref{vel} can be written as 
$$d_{V}(t)\leq D_{0}e^{-C(t-2\bar{\tau})},\quad \forall t\geq-\bar{\tau}.$$
Let us denote with $\mathbb{N}_{0}:=\mathbb{N}\cup\{0\}$.
\begin{lem}	\label{lemma2}
	For each $n\in \mathbb{N}_{0}$ we have that 
	\begin{equation}\label{dec}
		D_{n+1}\leq D_{n}.
	\end{equation}
\end{lem}
\begin{proof}
	Let $n\in\mathbb{N}_{0}$. If $D_{n+1}=0$, then of course $$D_{n}=\max_{i,j=1,\dots,N}\,\,\max_{s,t\in [n\bar{\tau}-\bar{\tau},n\bar{\tau}]}\lvert v_{i}(s)-v_{j}(t)\rvert\geq 0=D_{n+1}.$$ 
	So, suppose that $D_{n+1}>0$. Let $i,j=1,\dots,N$, $t_{1},t_{2}\in [n\bar{\tau},n\bar{\tau}+\bar{\tau}]$ be such that $$D_{n+1}=\lvert v_{i}(t_{1})-v_{j}(t_{2})\rvert.$$
	We set $$v=\frac{v_{i}(t_{1})-v_{j}(t_{2})}{\lvert v_{i}(t_{1})-v_{j}(t_{2})\rvert}.$$
	Then $v$ is a unit vector and, by using \eqref{inpro} with $T=n\bar{\tau}$ and the Cauchy-Schwarz inequality, we have that
	$$\begin{array}{l}
	\vspace{0.3cm}\displaystyle{D_{n+1}=\langle v_{i}(t_{1})-v_{j}(t_{2}),v\rangle=\langle v_{i}(t_{1}),v\rangle-\langle v_{j}(t_{2}),v\rangle}\\
	\vspace{0.3cm}\displaystyle{\hspace{0.9 cm}\leq \max_{l=1,\dots,N}\max_{s\in [n\bar{\tau}-\bar{\tau},n\bar{\tau}]}\langle v_{l}(s),v\rangle-\min_{l=1,\dots,N}\min_{s\in [n\bar{\tau}-\bar{\tau},n\bar{\tau}]}\langle v_{l}(s),v\rangle}\\
	\vspace{0.3cm}\displaystyle{\hspace{0.9 cm}\leq \max_{l,k=1,\dots,N}\max_{s,t\in [n\bar{\tau}-\bar{\tau},n\bar{\tau}]}\langle v_{l}(s)-v_{k}(t),v\rangle}\\
	\displaystyle{\hspace{0.9 cm}=\max_{l,k=1,\dots,N}\max_{s,t\in [n\bar{\tau}-\bar{\tau},n\bar{\tau}]} \lvert v_{l}(s)-v_{k}(t)\rvert=D_{n}.}
	\end{array}$$
\end{proof}
With an analogous argument, one can find a bound on $\vert v_i(t)\vert,$ uniform with respect to $t$ and $i=1,\dots,N$. Indeed, we have the following lemma.
\begin{lem}
	For every $i=1,\dots,N$, we have that
	\begin{equation}\label{bounvel}
	\lvert v_{i}(t)\rvert\leq R^{0}_{V}, \quad \forall t\geq-\bar{\tau},
	\end{equation}
	where
	$$R_{V}^{0}:=\max_{i=1,\dots,N}\,\,\max_{s\in [-\bar{\tau},0]}\lvert v_{i}(s)\rvert.$$
\end{lem}
\begin{proof}
	Let $i=1,\dots,N$ and $t\geq-\bar{\tau}$. Note that, if $\lvert v_{i}(t)\rvert =0$, then trivially $R^{0}_{V}\geq 0=\lvert v_{i}(t)\rvert $. So, we can assume $\lvert v_{i}(t)\rvert >0$. We define $$v=\frac{v_{i}(t)}{\lvert v_{i}(t)\rvert}.$$
	Then, $v$ is a unit vector for which we can write
	$$\lvert v_{i}(t)\rvert=\langle v_{i}(t),v\rangle. $$
	Since $t\geq-\bar{\tau}$, we can apply \eqref{inpro} for $T=0$ and we get $$\lvert v_{i}(t)\rvert\leq \max_{j=1,\dots,N}\max_{s\in [-\bar{\tau},0]}\langle v_{j}(s),v\rangle\leq \max_{j=1,\dots,N}\max_{s\in [-\bar{\tau},0]}\lvert v_{j}(s)\rvert\lvert v\rvert=\max_{j=1,\dots,N}\max_{s\in [-\bar{\tau},0]}\lvert v_{j}(s)\rvert=R^{0}_{V}.$$
	Thus, \eqref{bounvel} is fullfilled.
\end{proof}
\begin{lem}
	For every $i,j=1,\dots,N$, we get
	\begin{equation}\label{dist}
	\lvert x_{i}(t-\tau(t))-x_{j}(t)\rvert\leq 2\bar{\tau}R^{0}_{V}+4M_{X}^{0}+d_{X}(t-\bar{\tau}), \quad\forall t\geq0,
	\end{equation}
	where
	$$M^{0}_{X}:=\max_{i=1,\dots,N}\,\,\max_{s\in [-\bar{\tau},0]}\lvert x_{i}(s)\rvert.$$
\end{lem}
\begin{proof}
	Given $i,j=1,\dots,N$ and $t\geq0$, we have
	\begin{equation}\label{split}
	\begin{split}
	\lvert x_{i}(t-\tau(t))-x_{j}(t)\rvert&\leq \lvert x_{i}(t-\tau(t))-x_{i}(t-\bar{\tau})\rvert+\lvert x_{i}(t-\bar{\tau})-x_{j}(t-\bar{\tau})\rvert+\lvert x_{j}(t-\bar{\tau})-x_{j}(t)\rvert\\&\leq \lvert x_{i}(t-\tau(t))-x_{i}(t-\bar{\tau})\rvert +d_{X}(t-\bar{\tau})+\lvert x_{j}(t-\bar{\tau})-x_{j}(t)\rvert.
	\end{split}
	\end{equation}
	Now, assume $t> \bar{\tau}$. Then both $t-\bar{\tau},t-\tau(t)>0$ and from inequalities \eqref{delay} and \eqref{bounvel} we get 
	$$\lvert x_{i}(t-\tau(t))-x_{i}(t-\bar{\tau})\rvert=\left\lvert \int_{t-\bar{\tau}}^{t-\tau(t)}v_{i}(s)\,ds\right\rvert\leq \int_{t-\bar{\tau}}^{t-\tau(t)}\lvert v_{i}(s)\rvert ds$$$$\hspace{2cm}\leq R_{V}^{0}(t-\tau(t)-t+\bar{\tau})\leq \bar{\tau}R_{V}^{0},$$
	and 
	$$\lvert x_{j}(t-\bar{\tau})-x_{j}(t)\rvert=\left\lvert -\int_{t-\bar{\tau}}^{t}v_{j}(s)\,ds\right\rvert\leq \int_{t-\bar{\tau}}^{t}\lvert v_{j}(s)\rvert ds\leq \bar{\tau}R_{V}^{0}.$$
	Thus, \eqref{split} becomes $$\lvert x_{i}(t-\tau(t))-x_{j}(t)\rvert\leq 2\bar{\tau}R_{V}^{0}+d_{X}(t-\bar{\tau}).$$
	On the contrary, assume that $t\leq \bar{\tau}$. Then $t-\bar{\tau}\leq 0$ and from \eqref{delay} and \eqref{bounvel} we get
	$$\lvert x_{j}(t-\bar{\tau})-x_{j}(t)\rvert=\left\lvert x_{j}(t-\bar{\tau})-x_{j}(0)-\int_{0}^{t}v_{j}(s)\,ds\right\rvert$$$$\hspace{3cm}\leq \lvert x_{j}(t-\bar{\tau})-x_{j}(0)\rvert+\int_{0}^{t}\lvert v_{j}(s)\rvert ds$$$$\hspace{2.3cm}\leq 2M_{X}^{0}+tR_{V}^{0}\leq 2M_{X}^{0}+\bar{\tau}R_{V}^{0}.$$
	Note that our assumption, $t\leq\bar{\tau}$, does not imply that $t-\tau(t)\leq 0$. So we can distinguish two cases. \\If $t-\tau(t)>0$, then $$\lvert x_{i}(t-\tau(t))-x_{i}(t-\bar{\tau})\rvert=\left\lvert x_{i}(0)+\int_{0}^{t-\tau(t)}v_{i}(s)\,ds-x_{i}(t-\bar{\tau})\right\rvert$$$$\leq\lvert x_{i}(0)-x_{i}(t-\bar{\tau})\rvert +\int_{0}^{t-\tau(t)}\lvert v_{i}(s)\rvert ds$$$$\leq 2M_{X}^{0}+(t-\tau(t))R_{V}^{0}\leq 2M_{X}^{0}+tR_{V}^{0}\leq 2M_{X}^{0}+\bar{\tau}R_{V}^{0},$$
	and \eqref{split} becomes $$\lvert x_{i}(t-\tau(t))-x_{j}(t)\rvert\leq 4M_{X}^{0}+2\bar{\tau}R_{V}^{0}+d_{X}(t-\bar{\tau}).$$
	On the other hand, if $t-\tau(t)\leq 0$, we have 
	$$\lvert x_{i}(t-\tau(t))-x_{i}(t-\bar{\tau})\rvert\leq 2M_{X}^{0},$$
	and we can write 
	$$\lvert x_{i}(t-\tau(t))-x_{j}(t)\rvert\leq 4M_{X}^{0}+\bar{\tau}R_{V}^{0}+d_{X}(t-\bar{\tau}).$$
	We have so proved that, in all cases, $$\lvert x_{i}(t-\tau(t))-x_{j}(t)\rvert\leq 4M_{X}^{0}+2\bar{\tau}R_{V}^{0}+d_{X}(t-\bar{\tau}),$$
	which proves \eqref{dist}.
\end{proof}
In the following, given $t\geq -\bar{\tau}$, $i,j=1,\dots,N$ and a vector $v\in \RR^{d}$, we shall denote with
$$d_{V}^{(ij)_{v}}(t):=\langle v_{i}(t)-v_{j}(t),v\rangle.$$
The next lemma is analogous to Lemma 3.4 in \cite{Cartabia}. We give the proof for the reader's convenience.
\begin{lem}\label{disug}
	For all $i,j=1,\dots,N$, unit vector $v\in \RR^{d}$ and $n\in \mathbb{N}_{0}$, we have that 
	\begin{equation}\label{dvij}
		d_{V}^{(ij)_{v}}(t)\leq e^{-K(t-t_{0})}d_{V}^{(ij)_{v}}(t_{0})+(1-e^{-K(t-t_{0})})D_{n},
	\end{equation}
	for all $t\geq t_{0}\geq n\bar{\tau}$. Moreover, for each $n\in \mathbb{N}_{0}$ it holds
	\begin{equation}\label{n+1}
		D_{n+1}\leq e^{-K\bar{\tau}}d_{V}(n\bar{\tau})+(1-e^{-K\bar{\tau}})D_{n}.
	\end{equation}
\end{lem}
\begin{proof}
	Given $n\in \mathbb{N}_{0}$, for each $v\in \RR^{d}$ unit vector, let denote with $$M=\max_{l=1,\dots,N}\max_{s\in [n\bar{\tau}-\bar{\tau},n\bar{\tau}]}\langle v_{l}(s),v\rangle,$$
	$$m=\min_{l=1,\dots,N}\min_{s\in [n\bar{\tau}-\bar{\tau},n\bar{\tau}]}\langle v_{l}(s),v\rangle.$$
	Then $M-m\leq D_{n}$. We claim that, for all $i,j=1,\dots,N$, $t\geq t_{0}\geq n\bar{\tau}$,
	\begin{equation}\label{Mm}
		\begin{split}
		&\langle v_{i}(t),v\rangle\leq e^{-K(t-t_{0})}\langle v_{i}(t_{0}),v\rangle+(1-e^{-K(t-t_{0})})M,\\&\langle v_{j}(t),v\rangle\geq e^{-K(t-t_{0})}\langle v_{j}(t_{0}),v\rangle+(1-e^{-K(t-t_{0})})m.
		\end{split}
	\end{equation}
	So, fix $i,j=1,\dots,N$ and $t\geq t_{0}\geq n\bar{\tau}$. Then, being $t\geq 0$, we have
	\begin{equation}\label{derva}
		\begin{split}
		\frac{d}{dt}\langle v_{i}(t),v\rangle&=\sum_{l:l\neq i}a_{il}(t)\langle v_{l}(t-\tau(t))-v_{i}(t),v\rangle\\&=\sum_{l:l\neq i}a_{il}(t)(\langle v_{l}(t-\tau(t)),v\rangle-\langle v_{i}(t),v\rangle)
		\end{split}
	\end{equation}
	We recall that $a_{il}(t)=\frac{1}{N-1}\psi(\lvert x_{i}(t)-x_{l}(t-\tau(t))\rvert)$. Thus, being $\psi$ a bounded function, we can write $a_{il}(t)\leq \frac{K}{N-1}$. Furthermore, $t\geq n\bar{\tau}$, which implies that $t-\tau(t)\geq n\bar{\tau}-\bar{\tau}$. Then, by virtue of \eqref{inpro}, we have that $$m\leq\langle v_{k}(t-\tau(t)),v\rangle\leq M,\quad m\leq\langle v_{k}(t),v\rangle\leq M, \quad \forall k=1,\dots,N.$$ So, combining all these facts, \eqref{derva} becomes
	$$\frac{d}{dt}\langle v_{i}(t),v\rangle= \sum_{l:l\neq i}a_{il}(t)(\langle v_{l}(t-\tau(t)),v\rangle-M+M-\langle v_{i}(t),v\rangle)$$$$\hspace{1cm}\leq\sum_{l:l\neq i}a_{il}(t)(M-\langle v_{i}(t),v\rangle)$$$$\hspace{1.4cm}\leq\frac{K}{N-1}\sum_{l:l\neq i}(M-\langle v_{i}(t),v\rangle)$$$$ =K(M-\langle v_{i}(t),v\rangle).$$
	Then, from the Gronwall's inequality with $t\geq t_{0}$ we get
	$$\begin{array}{l}
	\vspace{0.2cm}\displaystyle{\langle  v_{i}(t),v\rangle\leq e^{-\int_{t_{0}}^{t}Kds}\langle v_{i}(t_{0}),v\rangle+\int_{t_{0}}^{t}KMe^{-({\int_{t_{0}}^{t}Kdv-\int_{t_{0}}^{s}}Kdv)}ds}\\
	\vspace{0.3cm}\displaystyle{\hspace{1.5cm}=e^{-K(t-t_{0})}\langle v_{i}(t_{0}),v\rangle+M e^{-K(t-t_{0})}(e^{K(t-t_{0})}-1)}\\
	\displaystyle{\hspace{1.5cm}=e^{-K(t-t_{0})}\langle v_{i}(t_{0}),v\rangle+(1-e^{-K(t-t_{0})})M.}
	\end{array}$$
	Hence, it holds \begin{equation}\label{prodM}
		\langle v_{i}(t),v\rangle\leq e^{-K(t-t_{0})}\langle v_{i}(t_{0}),v\rangle+(1-e^{-K(t-t_{0})})M,
	\end{equation}
	for every $i=1,\dots,N$, $t\geq t_{0}\geq n\bar{\tau}$ and unit vector $v\in \RR^{d}$, which proves the first inequality in \eqref{Mm}. 
	\\Now, to prove the second inequality in \eqref{Mm}, let $j=1,\dots,N$, $t\geq t_{0}\geq n\bar{\tau}$ and a unit vector $v\in \RR^{d}$. Then, we can apply \eqref{prodM} to the unit vector $-v\in \RR^{d}$ and we get $$\langle v_{j}(t),-v\rangle\leq e^{-K(t-t_{0})}\langle v_{j}(t_{0}),-v\rangle+(1-e^{-K(t-t_{0})})\left(\max_{l=1,\dots,N}\max_{s\in [n\bar{\tau}-\bar{\tau},n\bar{\tau}]}\langle v_{l}(s),-v\rangle\right),$$
	from which $$\langle v_{j}(t),v\rangle \geq e^{-K(t-t_{0})}\langle v_{j}(t_{0}), v\rangle+(1-e^{-K(t-t_{0})})\left(-\max_{l=1,\dots,N}\max_{s\in [n\bar{\tau}-\bar{\tau},n\bar{\tau}]}\langle v_{l}(s),-v\rangle\right)$$
	$$=e^{-K(t-t_{0})}\langle v_{j}(t_{0}), v\rangle+(1-e^{-K(t-t_{0})})m.$$
	Therefore \eqref{Mm} holds true. 
	\\Now, from \eqref{Mm}, for each $i,j=1,\dots,N$, $v\in \RR^{d}$ unit vector and $t\geq t_{0}\geq n\bar{\tau}$, we have
	$$\begin{array}{l}
	\vspace{0.3cm}\displaystyle{d_{V}^{(ij)_{v}}(t)=\langle v_{i}(t)-v_{j}(t),v\rangle=\langle v_{i}(t)(t),v\rangle-\langle v_{j}(t),v\rangle}\\
	\vspace{0.3cm}\hspace{0.5cm}\displaystyle{\leq e^{-K(t-t_{0})}\langle v_{i}(t_{0}),v\rangle+(1-e^{-K(t-t_{0})})M-e^{-K(t-t_{0})}\langle v_j(t_{0}), v\rangle-(1-e^{-K(t-t_{0})})m}\\
	\vspace{0.3cm}\hspace{0.5cm}\displaystyle{=e^{-K(t-t_{0})}\langle v_{i}(t_{0})-v_{j}(t_{0}),v\rangle+(1-e^{-K(t-t_{0})})(M-m)}\\
	\hspace{0.5cm}\displaystyle{=e^{-K(t-t_{0})}d_{V}^{(ij)_{v}}(t_{0})+(1-e^{-K(t-t_{0})})(M-m).}
	\end{array}$$
	Then, by recalling that $M-m\leq D_{n}$, we finally get
	$$d_{V}^{(ij)_{v}}(t)\leq e^{-K(t-t_{0})}d_{V}^{(ij)_{v}}(t_{0})+(1-e^{-K(t-t_{0})})D_{n},$$
	which proves \eqref{dvij}.
	\\Finally, we prove \eqref{n+1}. Let $i,j=1,\dots,N$ and $t_{1},t_{2}\in [n\bar{\tau},n\bar{\tau}+\bar{\tau}]$ be such that
	$$D_{n+1}=\lvert v_{i}(t_{1})-v_{j}(t_{2})\rvert.$$
	Note that, if $D_{n+1}=0$, then trivially $$e^{-K\bar{\tau}}d_{v}(n\bar{\tau})+(1-e^{-K\bar{\tau}})D_{n}\geq 0=D_{n+1}.$$
	So we can assume $D_{n+1}>0$ and we define the unit vector $$v=\frac{v_{i}(t_{1})-v_{j}(t_{2})}{\lvert v_{i}(t_{1})-v_{j}(t_{2})\rvert}.$$
	By applying \eqref{Mm} with $t_{0}=n\bar{\tau}\leq t_{1},t_{2}$, we get
	$$\begin{array}{l}
		\vspace{0.3cm}\displaystyle{\langle v_{i}(t_{1}),v\rangle\leq e^{-K(t_{1}-n\bar{\tau})}\langle v_{i}(n\bar{\tau}),v\rangle+(1-e^{-K(t_{1}-n\bar{\tau})})M}\\
		\vspace{0.3cm}\displaystyle{\hspace{1.7cm}=e^{-K(t_{1}-n\bar{\tau})}(\langle v_{i}(n\bar{\tau}),v\rangle-M)+M}\\
		\vspace{0.3cm}\displaystyle{\hspace{1.7cm}\leq e^{-K\bar{\tau}}(\langle v_{i}(n\bar{\tau}),v\rangle-M)+M}\\
		\displaystyle{\hspace{1.7cm}=e^{-K\bar{\tau}}\langle v_{i}(n\bar{\tau}),v\rangle+(1-e^{-K\bar{\tau}})M,}
	\end{array}$$
	where we used the fact that $t_{1}\leq n\bar{\tau}+\bar{\tau}$ and $\langle v_{i}(n\bar{\tau}),v\rangle-M\leq0$, and
	$$\begin{array}{l}
		\vspace{0.3cm}\displaystyle{\langle v_{j}(t_{2}),v\rangle\geq e^{-K(t_{2}-n\bar{\tau})}\langle v_{j}(n\bar{\tau}),v\rangle+(1-e^{-K(t_{2}-n\bar{\tau})})m}\\
		\vspace{0.3cm}\displaystyle{\hspace{1.7cm}=e^{-K(t_{2}-n\bar{\tau})}(\langle v_{j}(n\bar{\tau}),v\rangle-m)+m}\\
		\vspace{0.3cm}\displaystyle{\hspace{1.7cm}\geq e^{-K\bar{\tau}}(\langle v_{j}(n\bar{\tau}),v\rangle-m)+m}\\
		\displaystyle{\hspace{1.7cm}=e^{-K\bar{\tau}}\langle v_{j}(n\bar{\tau}),v\rangle+(1-e^{-K\bar{\tau}})m,}
	\end{array}$$
	where we used the fact that $t_{2}\leq n\bar{\tau}+\bar{\tau}$ and $\langle v_{j}(n\bar{\tau}),v\rangle-m\geq0$.
	\\As a consequence, it holds 
	$$\begin{array}{l}
	\vspace{0.3cm}\displaystyle{D_{n+1}=\langle v_{i}(t_{1})-v_{j}(t_{2}),v\rangle=\langle v_{i}(t_{1}),v\rangle-\langle v_{j}(t_{2}),v\rangle}\\
	\vspace{0.3cm}\displaystyle{\hspace{0.9cm}\leq e^{-K\bar{\tau}}\langle v_{i}(n\bar{\tau}),v\rangle+(1-e^{-K\bar{\tau}})M -e^{-K\bar{\tau}}\langle v_{j}(n\bar{\tau}),v\rangle-(1-e^{-K\bar{\tau}})m}\\
	\vspace{0.3cm}\displaystyle{\hspace{0.9cm}=e^{-K\bar{\tau}}\langle v_{i}(n\bar{\tau})-v_{j}(n\bar{\tau}),v\rangle+(1-e^{-K\bar{\tau}})(M-m)}\\
	\vspace{0.3cm}\displaystyle{\hspace{0.9cm}\leq e^{-K\bar{\tau}}\lvert v_{i}(n\bar{\tau})-v_{j}(n\bar{\tau})\rvert\lvert v\rvert+(1-e^{-K\bar{\tau}})(M-m)}\\
	\displaystyle{\hspace{0.9cm}\leq e^{-K\bar{\tau}}d_{V}(n\bar{\tau})+(1-e^{-K\bar{\tau}})D_{n},}
	\end{array}$$
	which concludes our proof.
\end{proof}
Now, we give the following definition.
\begin{defn}
	We define
	$$\phi(t):=\min\left\{e^{-K\bar{\tau}}\psi_{t},\frac{e^{-2K\bar{\tau}}}{\bar{\tau}}\right\},$$
	where
	$$\psi_{t}=\min\left\{\psi(r):r\in \left[0,2\bar{\tau}R_{V}^{0}+4M_{X}^{0}+\max_{s\in[-\bar{\tau},t] }d_{X}(s)\right]\right\},$$
	for all $t\geq -\bar{\tau}$.
\end{defn}
By definition, being $\psi$ a positive function, we have that $\psi_{t}>0$, for all $t\geq -\bar{\tau}$. Thus, the function $\phi$ is positive too.
\begin{oss}
	Let us note that from estimate \eqref{dist}, for all $t\geq 0$ and $i,j=1,\dots,N$, it holds that
	$$\psi(\lvert x_{i}(t)-x_{j}(t-\tau(t))\rvert)\geq \psi_{t-\bar{\tau}},$$
	from which
	\begin{equation}\label{lowerbound}
		\psi(\lvert x_{i}(t)-x_{j}(t-\tau(t))\rvert)\geq e^{K\bar{\tau}}\phi(t-\bar{\tau}).
	\end{equation}
\end{oss}
\begin{lem}\label{lemma3}
	For each integer $n\geq 2$, we have that
	\begin{equation}\label{n-2}
		D_{n+1}\leq \left(1-e^{-K\bar{\tau}}\int_{n\bar{\tau}-2\bar{\tau}}^{n\bar{\tau}-\bar{\tau}}\phi(s)ds\right)D_{n-2}.
	\end{equation}
\end{lem}
\begin{proof}
	We first show that, for each $n\geq2$,
	\begin{equation}\label{dv}
		d_{V}(n\bar{\tau})\leq \left(1-\int_{n\bar{\tau}-2\bar{\tau}}^{n\bar{\tau}-\bar{\tau}}\phi(s)ds\right)D_{n-2}.
	\end{equation}
	To this aim, let $n\geq2$. Note that, if $d_{V}(n\bar{\tau})=0$, by definition of $\phi$ we have that
	$$\left(1-\int_{n\bar{\tau}-2\bar{\tau}}^{n\bar{\tau}-\bar{\tau}}\phi(s)ds\right)D_{n-2}\geq \left(1-\int_{n\bar{\tau}-2\bar{\tau}}^{n\bar{\tau}-\bar{\tau}}\frac{e^{-2K\bar{\tau}}}{\bar{\tau}}ds\right)D_{n-2}$$
	$$=\left(1-\frac{e^{-2K\bar{\tau}}}{\bar{\tau}}(n\bar{\tau}-\bar{\tau}-n\bar{\tau}+2\bar{\tau})\right)D_{n-2}$$$$=\left(1-e^{-2K\bar{\tau}}\right)D_{n-2}\geq 0=d_{V}(n\bar{\tau}).$$
	So we can assume $d_{V}(n\bar{\tau})>0$. Moreover, let $i,j=1,\dots,N$ be such that $$d_{V}(n\bar{\tau})=\lvert v_{i}(n\bar{\tau})-v_{j}(n\bar{\tau})\rvert.$$
	We set $$v=\frac{v_{i}(n\bar{\tau})-v_{j}(n\bar{\tau})}{\lvert v_{i}(n\bar{\tau})-v_{j}(n\bar{\tau})\rvert}.$$
	Then $v$ is a unit vector for which we can write $$d_{V}(n\bar{\tau})=\langle v_{i}(n\bar{\tau})-v_{j}(n\bar{\tau}),v\rangle=d_{V}^{(ij)_{v}}(n\bar{\tau}).$$
	At this point, we distinguish two cases.
	\\\par\textit{Case I.} Assume that there exists $t_{0}\in [n\bar{\tau}-2\bar{\tau},n\bar{\tau}]$ such that $d_{V}^{(ij)_{v}}(t_{0})<0$. Note that
	$$\left(1-\int_{n\bar{\tau}-2\bar{\tau}}^{n\bar{\tau}-\bar{\tau}}\phi(s)ds\right)\geq 1-e^{-2K\bar{\tau}}.$$
	Then, by using \eqref{dvij} with $n\bar{\tau}\geq t_{0}\geq n\bar{\tau}-2\bar{\tau}$, we have 
	$$d_{V}^{(ij)_{v}}(n\bar{\tau})\leq e^{-K(n\bar{\tau}-t_{0})}d_{V}^{(ij)_{v}}(t_{0})+(1-e^{-K(n\bar{\tau}-t_{0})})D_{n-2}$$$$<(1-e^{-K(n\bar{\tau}-t_{0})})D_{n-2}\leq (1-e^{-2K\bar{\tau}})D_{n-2}\leq \left(1-\int_{n\bar{\tau}-2\bar{\tau}}^{n\bar{\tau}-\bar{\tau}}\phi(s)ds\right)D_{n-2}.$$
	\\\par \textit{Case II.}  Assume that $d_{V}^{(ij)_{v}}(t)\geq0$, for every $t\in [n\bar{\tau}-2\bar{\tau},n\bar{\tau}]$. We set
	$$M=\max_{l=1,\dots,N}\max_{s\in [n\bar{\tau}-2\bar{\tau},n\bar{\tau}-\bar{\tau}]}\langle v_{l}(s),v\rangle,$$
	$$m=\min_{l=1,\dots,N}\min_{s\in [n\bar{\tau}-2\bar{\tau},n\bar{\tau}-\bar{\tau}]}\langle v_{l}(s),v\rangle.$$
	Then, $M-m\leq D_{n-1}$. Notice that, from \eqref{lowerbound}, for each $l,k=1,\dots,N$ and $t\geq 0$,  
	\begin{equation}\label{pesi1}
		a_{lk}(t)\geq \frac{e^{K\bar{\tau}}\phi(t-\bar{\tau})}{N-1}.
	\end{equation}
	Thus, for every $t\in [n\bar{\tau}-\bar{\tau},n\bar{\tau}]$, it comes that
	$$\begin{array}{l}
	\displaystyle{\frac{d}{dt}d_{V}^{(ij)_{v}}(t)=\sum_{l:l\neq i}a_{il}(t)\langle v_{l}(t-\tau(t))-v_{i}(t),v\rangle+\sum_{l:l\neq i}a_{jl}(t)\langle v_{j}(t)-v_{l}(t-\tau(t)),v\rangle}\\
	\displaystyle{\hspace{1.8cm}=\sum_{l:l\neq i}a_{il}(t)(\langle v_{l}(t-\tau(t)),v\rangle-M+M-\langle v_{i}(t),v\rangle)}\\
	\vspace{0.3cm}\displaystyle{\hspace{2.1cm}+\sum_{l:l\neq i}a_{jl}(t)(\langle v_{j}(t),v\rangle-m+m-\langle v_{l}(t-\tau(t)),v\rangle)}\\
	\displaystyle{\hspace{6 cm}:=S_{1}+S_{2}.}
	\end{array}$$
	We recall that $\psi$ is bounded and that, from \eqref{inpro}, $$m\leq \langle v_{k}(s),v\rangle\leq M,\quad \forall s\geq n\bar{\tau}-2\bar{\tau},\forall k=1,\dots,N.$$ Combining these facts with \eqref{pesi1}, for every $t\in [n\bar{\tau}-\bar{\tau},n\bar{\tau}]$, it holds that $t,t-\tau(t)\geq n\bar{\tau}-2\bar{\tau}$ and we can write
	$$\begin{array}{l}
	\vspace{0.2cm}\displaystyle{S_{1}= \sum_{l:l\neq i}a_{il}(t)(\langle v_{l}(t-\tau(t)),v\rangle-M)+\sum_{l:l\neq i}a_{il}(t)(M-\langle v_{i}(t),v\rangle)}\\
	\displaystyle{\hspace{0.4cm}\leq \frac{e^{K\bar{\tau}}\phi(t-\bar{\tau})}{N-1}\sum_{l:l\neq i}(\langle v_{l}(t-\tau(t)),v\rangle-M)+\frac{K}{N-1}\sum_{l:l\neq i}(M-\langle v_{i}(t),v\rangle)}\\
	\displaystyle{\hspace{0.4cm}=\frac{e^{K\bar{\tau}}\phi(t-\bar{\tau})}{N-1}\sum_{l:l\neq i}(\langle v_{l}(t-\tau(t)),v\rangle-M)+K(M-\langle v_{i}(t),v\rangle),}
	\end{array}$$
	and
	$$\begin{array}{l}
		\vspace{0.2cm}\displaystyle{S_{2}=\sum_{l:l\neq j}a_{jl}(t)(\langle v_{j}(t),v\rangle-m)+\sum_{l:l\neq j}a_{jl}(t)(m-\langle v_{l}(t-\tau(t)),v\rangle)}\\
		\displaystyle{\hspace{0.4cm}\leq \frac{K}{N-1}\sum_{l:l\neq j}(\langle v_{j}(t),v\rangle-m)+\frac{e^{K\bar{\tau}}\phi(t-\bar{\tau})}{N-1}\sum_{l:l\neq j}(m-\langle v_{l}(t-\tau(t)),v\rangle)}\\
		\displaystyle{\hspace{0.4cm}=K(\langle v_{j}(t),v\rangle-m)+\frac{e^{K\bar{\tau}}\phi(t-\bar{\tau})}{N-1}\sum_{l:l\neq j}(m-\langle v_{l}(t-\tau(t)),v\rangle).}
	\end{array}$$
	Hence, we get
	$$\begin{array}{l}
	\vspace{0.2cm}\displaystyle{S_{1}+S_{2}\leq K(M-\langle v_{i}(t),v\rangle+\langle v_{j}(t),v\rangle-m)}\\
	\displaystyle{\hspace{1.5cm}+\frac{e^{K\bar{\tau}}\phi(t-\bar{\tau})}{N-1}\sum_{l:l\neq i,j}(\langle v_{l}(t-\tau(t)),v\rangle-M+m-\langle v_{l}(t-\tau(t)),v\rangle)}\\
	\vspace{0.2cm}\displaystyle{\hspace{1.5cm}+\frac{e^{K\bar{\tau}}\phi(t-\bar{\tau})}{N-1}(\langle v_{j}(t-\tau(t)),v\rangle-M+m-\langle v_{i}(t-\tau(t)),v\rangle)}\\
	\vspace{0.2cm}\displaystyle{\hspace{1.2cm}=K(M-m-d_{V}^{(ij)_{v}}(t))+\frac{e^{K\bar{\tau}}\phi(t-\bar{\tau})}{N-1}(N-2)(m-M)}\\
	\displaystyle{\hspace{1.5cm}+\frac{e^{K\bar{\tau}}\phi(t-\bar{\tau})}{N-1}(m-M-d_{V}^{(ij)_{v}}(t-\tau(t))).}
	\end{array}$$
	Note that, being $t\in [n\bar{\tau}-\bar{\tau},n\bar{\tau}]$, it holds that $t-\tau(t)\in [n\bar{\tau}-2\bar{\tau},n\bar{\tau}]$. Therefore, from our assumption, we have $d_{V}^{(ij)_{v}}(t-\tau(t))\geq 0$, from which follows that  $$\frac{e^{K\bar{\tau}}\phi(t-\bar{\tau})}{N-1}(m-M-d_{V}^{(ij)_{v}}(t-\tau(t))\leq\frac{e^{K\bar{\tau}}\phi(t-\bar{\tau})}{N-1}(m-M).$$
	Thus, taking into account of the fact that $M-m\leq D_{n-1}$, we get  
	$$\begin{array}{l}
	\vspace{0.2cm}\displaystyle{\frac{d}{dt}d_{V}^{(ij)_{v}}(t)\leq K(M-m-d_{V}^{(ij)_{v}}(t))+\frac{e^{K\bar{\tau}}\phi(t-\bar{\tau})}{N-1}(N-2)(m-M)+\frac{e^{K\bar{\tau}}\phi(t-\bar{\tau})}{N-1}(m-M)}\\
	\vspace{0.3cm}\displaystyle{\hspace{1.8cm}=K(M-m-d_{V}^{(ij)_{v}}(t))+e^{K\bar{\tau}}\phi(t-\bar{\tau})(m-M)}\\
	\vspace{0.3cm}\displaystyle{\hspace{1.8cm}=(K-e^{K\bar{\tau}}\phi(t-\bar{\tau}))(M-m)-Kd_{V}^{(ij)_{v}}(t)}\\
	\displaystyle{\hspace{1.8cm}\leq (K-e^{K\bar{\tau}}\phi(t-\bar{\tau}))D_{n-1}-Kd_{V}^{(ij)_{v}}(t),}
	\end{array}$$
	for every $t\in [n\bar{\tau}-\bar{\tau},n\bar{\tau}]$.
	Then, from the Gronwall's inequality, for every $t\in [n\bar{\tau}-\bar{\tau},n\bar{\tau}]$, we have
	$$\begin{array}{l}
	\vspace{0.2cm}\displaystyle{d_{V}^{(ij)_{v}}(t)\leq e^{-\int_{n\bar{\tau}-\bar{\tau}}^{t}Kds}d_{V}^{(ij)_{v}}(n\bar{\tau}-\bar{\tau})+D_{n-1}\int_{n\bar{\tau}-\bar{\tau}}^{t}(K-e^{K\bar{\tau}}\phi(s-\bar{\tau}))e^{-\left(\int_{n\bar{\tau}-\bar{\tau}}^{t}Kdv-\int_{n\bar{\tau}-\bar{\tau}}^{s}Kdv\right)}ds}\\
	\vspace{0.2cm}\displaystyle{\hspace{1.4cm}=e^{-K(t-n\bar{\tau}+\bar{\tau})}d_{V}^{(ij)_{v}}(n\bar{\tau}-\bar{\tau})+D_{n-1}\int_{n\bar{\tau}-\bar{\tau}}^{t}(K-e^{K\bar{\tau}})\phi(s-\bar{\tau})e^{-K(t-s)}ds}\\
	\vspace{0.2cm}\displaystyle{\hspace{1.4cm}=e^{-K(t-n\bar{\tau}+\bar{\tau})}d_{V}^{(ij)_{v}}(n\bar{\tau}-\bar{\tau})+D_{n-1}\left(e^{-Kt}[e^{Ks}]_{n\bar{\tau}-\bar{\tau}}^{t}-e^{K\bar{\tau}}\int_{n\bar{\tau}-\bar{\tau}}^{t}e^{-K(t-s)}\phi(s-\bar{\tau})ds\right)}\\
	\displaystyle{\hspace{1.4cm}=e^{-K(t-n\bar{\tau}+\bar{\tau})}d_{V}^{(ij)_{v}}(n\bar{\tau}-\bar{\tau})+D_{n-1}\left(1-e^{-K(t-n\bar{\tau}+\bar{\tau})}-e^{K\bar{\tau}}\int_{n\bar{\tau}-\bar{\tau}}^{t}e^{-K(t-s)}\phi(s-\bar{\tau})ds\right).}
	\end{array}$$
	In particular, for $t=n\bar{\tau}$ it holds $$\begin{array}{l}
	\vspace{0.3cm}\displaystyle{d_{V}^{(ij)_{v}}(n\bar{\tau})\leq e^{-K\bar{\tau}}d_{V}^{(ij)_{v}}(n\bar{\tau}-\bar{\tau})+D_{n-1}\left(1-e^{-K\bar{\tau}}-e^{K\bar{\tau}}\int_{n\bar{\tau}-\bar{\tau}}^{n\bar{\tau}}e^{-K(n\bar{\tau}-s)}\phi(s-\bar{\tau})ds\right)}\\
	\displaystyle{\hspace{1.7cm}\leq e^{-K\bar{\tau}}d_{V}^{(ij)_{v}}(n\bar{\tau}-\bar{\tau})+D_{n-1}\left(1-e^{-K\bar{\tau}}-e^{K\bar{\tau}}\int_{n\bar{\tau}-\bar{\tau}}^{n\bar{\tau}}\phi(s-\bar{\tau})ds\right).}
	\end{array}$$
	Notice that $d_{V}^{(ij)_{v}}(n\bar{\tau}-\bar{\tau})\leq D_{n-1}$ and that $$e^{K\bar{\tau}}\int_{n\bar{\tau}-\bar{\tau}}^{n\bar{\tau}}\phi(s-\bar{\tau})ds\geq \int_{n\bar{\tau}-\bar{\tau}}^{n\bar{\tau}}\phi(s-\bar{\tau})ds.$$
	So we can write
	$$d_{V}^{(ij)_{v}}(n\bar{\tau})\leq e^{-K\bar{\tau}}D_{n-1}+D_{n-1}\left(1-e^{-K\bar{\tau}}-\int_{n\bar{\tau}-\bar{\tau}}^{n\bar{\tau}}\phi(s-\bar{\tau})ds\right)$$$$= D_{n-1}\left(1-\int_{n\bar{\tau}-\bar{\tau}}^{n\bar{\tau}}\phi(s-\bar{\tau})ds\right).$$
	Then, with a change of variable, we get
	$$d_{V}^{(ij)_{v}}(n\bar{\tau})\leq D_{n-1}\left(1-\int_{n\bar{\tau}-2\bar{\tau}}^{n\bar{\tau}-\bar{\tau}}\phi(s)ds\right),$$and, being $D_{n-1}\leq D_{n-2}$, we can conclude that $$d_{V}^{(ij)_{v}}(n\bar{\tau})\leq\left(1-\int_{n\bar{\tau}-2\bar{\tau}}^{n\bar{\tau}-\bar{\tau}}\phi(s)ds\right)D_{n-2}.$$
	Therefore, \eqref{dv} holds true.
	\\Now, we are able to prove \eqref{n-2}. Indeed, for each $n\geq 2$, from \eqref{n+1} and \eqref{n-2}, it immediately follows that
	$$D_{n+1}\leq e^{-K\bar{\tau}}d_{V}(n\bar{\tau})+(1-e^{-K\bar{\tau}})D_{n}$$$$\hspace{4.5cm}\leq e^{-K\bar{\tau}} \left(1-\int_{n\bar{\tau}-2\bar{\tau}}^{n\tau_{1}-\bar{\tau}}\phi(s)ds\right)D_{n-2}+(1-e^{-K\bar{\tau}})D_{n-2}$$$$\hspace{1.4cm}=\left(1-e^{-K\bar{\tau}}\int_{n\bar{\tau}-2\bar{\tau}}^{n\bar{\tau}-\bar{\tau}}\phi(s)ds\right)D_{n-2}.$$
\end{proof}
\section{Proof of Theorem 1.1}
\setcounter{equation}{0}

\begin{proof}
	Let $\{(x_{i},v_{i})\}_{i=1,\dots,N}$  be solution to \eqref{csp} under the initial conditions \eqref{incond}. Following Rodriguez Cartabia \cite{Cartabia}, we introduce the function $\mathcal{D}:[-\bar{\tau},\infty)\rightarrow [0,\infty),$ defined as 
		$$\mathcal{D}(t):=\begin{cases}
		D_{0}, &t\in [-\bar{\tau},2\bar{\tau}]\\\mathcal{D}(n\bar{\tau})\left(1-e^{-K\bar{\tau}}\int_{n\bar{\tau}}^{t}\phi(s)ds\right)^{\frac{1}{3}}, &t\in (n\bar{\tau},n\bar{\tau}+\bar{\tau}],\,n\geq 2
		\end{cases}.$$
	By construction, $\mathcal{D}$ is continuous and nonincreasing. Moreover, we claim that \begin{equation}\label{boundIn}
		D_{n}\leq \mathcal{D}(t),
	\end{equation} for all $n\in \mathbb{N}_{0}$ and $t\in [-\bar{\tau},n\bar{\tau}]$.
	To prove this, we first show that, for each $n\geq 3$, \begin{equation}\label{int}
		1-e^{-K\bar{\tau}}\int_{n\bar{\tau}-2\bar{\tau}}^{n\bar{\tau}-\bar{\tau}}\phi(s)ds\leq \frac{\mathcal{D}(n\bar{\tau}+\bar{\tau})}{\mathcal{D}(n\bar{\tau}-2\bar{\tau})}.
	\end{equation}
	So, let $n\geq 3$. We split $$1-e^{-K\bar{\tau}}\int_{n\bar{\tau}-2\bar{\tau}}^{n\bar{\tau}-\bar{\tau}}\phi(s)ds$$$$=\left(1-e^{-K\bar{\tau}}\int_{n\bar{\tau}-2\bar{\tau}}^{n\bar{\tau}-\bar{\tau}}\phi(s)ds\right)^{\frac{1}{3}}\left(1-e^{-K\bar{\tau}}\int_{n\bar{\tau}-2\bar{\tau}}^{n\bar{\tau}-\bar{\tau}}\phi(s)ds\right)^{\frac{1}{3}}\left(1-e^{-K\bar{\tau}}\int_{n\bar{\tau}-2\bar{\tau}}^{n\bar{\tau}-\bar{\tau}}\phi(s)ds\right)^{\frac{1}{3}}.$$
	Now, it is easy to see that $\phi$ is a nonincreasing function. Thus, for each $m\geq n$, $$\int_{n\bar{\tau}-2\bar{\tau}}^{n\bar{\tau}-\bar{\tau}}\phi(s)ds\geq \int_{m\bar{\tau}-2\bar{\tau}}^{m\bar{\tau}-\bar{\tau}}\phi(s)ds.$$
	So we can write $$1-e^{-K\bar{\tau}}\int_{n\bar{\tau}-2\bar{\tau}}^{n\bar{\tau}-\bar{\tau}}\phi(s)ds$$$$\leq \left(1-e^{-K\bar{\tau}}\int_{n\bar{\tau}-2\bar{\tau}}^{n\bar{\tau}-\bar{\tau}}\phi(s)ds\right)^{\frac{1}{3}}\left(1-e^{-K\bar{\tau}}\int_{n\bar{\tau}-\bar{\tau}}^{n\bar{\tau}}\phi(s)ds\right)^{\frac{1}{3}}\left(1-e^{-K\bar{\tau}}\int_{n\bar{\tau}}^{n\bar{\tau}+\bar{\tau}}\phi(s)ds\right)^{\frac{1}{3}}$$$$\,\,\,=\frac{\mathcal{D}(n\bar{\tau}-\bar{\tau})}{\mathcal{D}(n\bar{\tau}-2\bar{\tau})}\,\frac{\mathcal{D}(n\bar{\tau})}{\mathcal{D}(n\bar{\tau}-\bar{\tau})}\,\frac{\mathcal{D}(n\bar{\tau}+\bar{\tau})}{\mathcal{D}(n\bar{\tau})}=\frac{\mathcal{D}(n\bar{\tau}+\bar{\tau})}{\mathcal{D}(n\bar{\tau}-2\bar{\tau})}\,,$$
	from which \eqref{int} holds true.
	\\At this point, we are able to prove \eqref{boundIn}. By induction, if $n\leq 2$, from Lemma \ref{lemma2} we can immediately say that $$D_{n}\leq D_{0}=\mathcal{D}(t), $$
	for all $t\in [-\bar{\tau},2\bar{\tau}]$. So we can assume that \eqref{boundIn} holds for each $2< m\leq n$ and prove it for $n+1$. From the induction hypothesis and by using again Lemma \ref{lemma2}, we have $$D_{n+1}\leq D_{n}\leq \mathcal{D}(t),$$
	for all $t\in [-\bar{\tau},n\bar{\tau}]$. On the other hand, for all $t\in (n\bar{\tau},n\bar{\tau}+\bar{\tau}]$, being $n>2$, from \eqref{n-2} we get $$D_{n+1}\leq \left(1-e^{-K\bar{\tau}}\int_{n\bar{\tau}-2\bar{\tau}}^{n\bar{\tau}-\bar{\tau}}\phi(s)ds\right)D_{n-2}.$$
	From the induction hypothesis or from the base case, $D_{n-2}\leq \mathcal{D}(t)$, for each $t\in [-\bar{\tau},n\bar{\tau}-2\bar{\tau}]$ So, in particular, $D_{n-2}\leq \mathcal{D}(n\bar{\tau}-2\bar{\tau})$. Therefore, combining this with \eqref{int} and with the fact that $\mathcal{D}$ is nonincreasing, we have that
	$$D_{n+1}\leq \frac{\mathcal{D}(n\bar{\tau}+\bar{\tau})}{\mathcal{D}(n\bar{\tau}-2\bar{\tau})}\,D_{n-2}\leq \frac{\mathcal{D}(n\bar{\tau}+\bar{\tau})}{\mathcal{D}(n\bar{\tau}-2\bar{\tau})}\,\mathcal{D}(n\bar{\tau}-2\bar{\tau})=\mathcal{D}(n\bar{\tau}+\bar{\tau})\leq \mathcal{D}(t),$$
	for all $t\in (n\bar{\tau},n\bar{\tau}+\bar{\tau}]$, which proves \eqref{boundIn}. 
	\\Now, notice that, for almost all time  $$\frac{d}{dt}\max_{s\in [-\bar{\tau},t]}d_{X}(s)\leq \left\lvert \frac{d}{dt}d_{X}(t)\right\rvert,$$
	since $\max_{s\in [-\bar{\tau},t]}d_{X}(s)$ is constant or increases like $d_{X}(t)$. Moreover, for almost all time $$\left\lvert \frac{d}{dt}d_{X}(t)\right\rvert\leq d_{V}(t).$$
	To see this, let $i,j=1,\dots,N$ be such that $d_{X}(t)=\lvert x_{i}(t)-xj(t)\rvert$. Obviously, if $\left\lvert\frac{d}{dt}d_{X}(t)\right\rvert=0$, then $$\left\lvert \frac{d}{dt}d_{X}(t)\right\rvert=0\leq d_{V}(t).$$
	So we can assume $\left\lvert\frac{d}{dt}d_{X}(t)\right\rvert>0$. Notice that $$\frac{d}{dt}( d_{X}(t))^{2}=\frac{d}{dt}\lvert x_{i}(t)-x_{j}(t)\rvert^{2}=2\lvert x_{i}(t)-x_{j}(t)\rvert\,\frac{d}{dt}\lvert x_{i}(t)-x_{j}(t)\rvert$$$$=2\lvert x_{i}(t)-x_{j}(t)\rvert\,\frac{d}{dt}d_{X}(t),$$
	with $\lvert x_{i}(t)-x_{j}(t)\rvert>0$, since otherwise $d_{X}(\cdot)$ wouldn't be differentiable at t. Also, $$\frac{d}{dt}(d_{X}(t))^{2}=2\langle v_{i}(t)-v_{j}(t), x_{i}(t)-x_{j}(t)\rangle,$$
	so that $$\lvert x_{i}(t)-x_{j}(t)\rvert\,\frac{d}{dt}d_{X}(t)=\langle v_{i}(t)-v_{j}(t), x_{i}(t)-x_{j}(t)\rangle.$$
	Thus,$$\lvert x_{i}(t)-x_{j}(t)\rvert\,\left\lvert\frac{d}{dt}d_{X}(t)\right\rvert=\lvert\langle v_{i}(t)-v_{j}(t), x_{i}(t)-x_{j}(t)\rangle \rvert\leq \lvert v_{i}(t)-v_{j}(t)\rvert\lvert x_{i}(t)-x_{j}(t)\rvert,$$
	from which, dividing by $\lvert x_{i}(t)-x_{j}(t)\rvert$, we get $$\left\lvert\frac{d}{dt}d_{X}(t)\right\rvert\leq \lvert v_{i}(t)-v_{j}(t)\rvert\leq  d_{V}(t).$$
	Therefore, for almost all time
	\begin{equation}\label{max}
		\frac{d}{dt}\max_{s\in [-\bar{\tau},t]}d_{X}(s)\leq \left\lvert \frac{d}{dt}d_{X}(t)\right\rvert\leq  d_{V}(t).
	\end{equation}
	Next, let $\mathcal{L}:[-\bar{\tau},\infty)\rightarrow [0,\infty)$ be the function given by $$\mathcal{L}(t):=\mathcal{D}(t)+\frac{e^{-K\bar{\tau}}}{3}\int_{0}^{2\bar{\tau}R_{V}^{0}+4M_{X}^{0}+\underset{s\in [-\bar{\tau},t]}{\max}d_{X}(s)}\min\left\{e^{-K\bar{\tau}}\min_{\sigma\in [0,r]}\psi(\sigma),\frac{e^{-2K\bar{\tau}}}{\bar{\tau}}\right\}dr,$$
	for all $t\geq-\bar{\tau}$.	By definition, $\mathcal{L}$ is continuous. In addition, for each $n\geq2$ and for $a.\,e.\,t\in(n\bar{\tau},n\bar{\tau}+\bar{\tau}) $, we have that $$\frac{d}{dt}\mathcal{L}(t)=\frac{d}{dt}\mathcal{D}(t)+\frac{e^{-K\bar{\tau}}}{3}\min\left\{e^{-K\bar{\tau}}\psi_{t},\frac{e^{-2K\bar{\tau}}}{\bar{\tau}}\right\}\frac{d}{dt}\underset{s\in [-\bar{\tau},t]}{\max}d_{X}(s)$$$$=\frac{d}{dt}\mathcal{D}(t)+\frac{e^{-K\bar{\tau}}}{3}\phi(t)\frac{d}{dt}\underset{s\in [-\bar{\tau},t]}{\max}d_{X}(s),$$
	and from \eqref{max} we get $$\frac{d}{dt}\mathcal{L}(t)\leq \frac{d}{dt}\mathcal{D}(t)+\frac{e^{-K\bar{\tau}}}{3}\phi(t)d_{V}(t).$$
	Now, for $a.\,e.\,t\in(n\bar{\tau},n\bar{\tau}+\bar{\tau})$, with $n\geq 2$, we compute $$\frac{d}{dt}\mathcal{D}(t)=-\frac{1}{3}\mathcal{D}(n\bar{\tau})\left(1-e^{-K\bar{\tau}}\int_{n\bar{\tau}}^{t}\phi(s)ds\right)^{-\frac{2}{3}}e^{-K\bar{\tau}}\phi(t).$$
	Thus, for each $n\geq 2$ and for $a.\,e.\,t\in(n\bar{\tau},n\bar{\tau}+\bar{\tau})$, $$\frac{d}{dt}\mathcal{L}(t)\leq\frac{e^{-K\bar{\tau}}}{3}\phi(t)\left(d_{V}(t)-\frac{\mathcal{D}(n\bar{\tau})}{\left(1-e^{-K\bar{\tau}}\int_{n\bar{\tau}}^{t}\phi(s)ds\right)^{\frac{2}{3}}}\right)$$$$\leq \frac{e^{-K\bar{\tau}}}{3}\phi(t)(d_{V}(t)-\mathcal{D}(n\bar{\tau})).$$
	Lastly, we can note that $d_{V}(t)\leq \mathcal{D}(n\bar{\tau})$, since $d_{V}(t)\leq D_{n+1}$ and $D_{n+1}\leq \mathcal{D}(n\bar{\tau})$ from inequality \eqref{boundIn}. Then, we get \begin{equation}\label{negder}
		\frac{d}{dt}\mathcal{L}(t)\leq0,
	\end{equation}
	for $a.\,e.\,t\in (n\bar{\tau},n\bar{\tau}+\bar{\tau})$ and for each $n\geq2$. Integrating \eqref{negder} over $(2\bar{\tau},t)$ for $t>2\bar{\tau}$ it comes that \begin{equation}\label{2tau1}
		\mathcal{L}(t)\leq \mathcal{L}(2\bar{\tau}).
	\end{equation}
	Therefore, from \eqref{2tau1}, it holds \begin{equation}\label{lim}
		\frac{e^{-K\bar{\tau}}}{3}\int_{0}^{\bar{\tau}R_{V}^{0}+2M_{X}^{0}+\underset{s\in [-\bar{\tau},t]}{\max}d_{X}(s)}\min\left\{e^{-K\bar{\tau}}\min_{\sigma\in [0,r]}\psi(\sigma),\frac{e^{-2K\bar{\tau}}}{\bar{\tau}}\right\}dr\leq\mathcal{L}(2\bar{\tau}),
	\end{equation}
	for all $t\geq2\bar{\tau}$. Letting $t\to \infty$ in \eqref{lim}, we finally get \begin{equation}\label{lim2}
		\frac{e^{-K\bar{\tau}}}{3}\int_{0}^{\bar{\tau}R_{V}^{0}+2M_{X}^{0}+\underset{s\in [-\bar{\tau},\infty)}{\sup}d_{X}(s)}\min\left\{e^{-K\bar{\tau}}\min_{\sigma\in [0,r]}\psi(\sigma),\frac{e^{-2K\bar{\tau}}}{\bar{\tau}}\right\}dr\leq \mathcal{L}(2\bar{\tau}).
	\end{equation} 
	Finally, since the function $\psi$ satisfies property \eqref{infint}, from \eqref{lim2}, we can conclude that there exists a positive constant $d^{*}$ such that \begin{equation}\label{firstcond}
		\bar{\tau}R_{V}^{0}+2M_{X}^{0}+\underset{s\in [-\bar{\tau},\infty)}{\sup}d_{X}(s)\leq d^{*}.
	\end{equation}
	Indeed, assume by contradiction that \begin{equation}\label{assurdo}
		\bar{\tau}R_{V}^{0}+2M_{X}^{0}+\underset{s\in [-\bar{\tau},\infty)}{\sup}d_{X}(s)=+\infty.
	\end{equation}
Then, equation \eqref{lim2} reads as
\begin{equation}\label{assurdo2}
	\int_{0}^{+\infty}\min\left\{e^{-K\bar{\tau}}\min_{\sigma\in [0,r]}\psi(\sigma),\frac{e^{-2K\bar{\tau}}}{\bar{\tau}}\right\}dr\leq \mathcal{L}(2\bar{\tau})
\end{equation}
Now, two different situations can occur.
\\Case I) Assume that, for all $r\in [0,+\infty)$, $$\frac{e^{-2K\bar{\tau}}}{\bar{\tau}}\leq e^{-K\bar{\tau}}\min_{\sigma\in [0,r]}\psi(\sigma) .$$
Thus, $$\int_{0}^{+\infty}\min\left\{e^{-K\bar{\tau}}\min_{\sigma\in [0,r]}\psi(\sigma),\frac{e^{-2K\bar{\tau}}}{\bar{\tau}}\right\}dr=\int_{0}^{+\infty}\frac{e^{-2K\bar{\tau}}}{\bar{\tau}}dr=+\infty,$$
which is in contradiction with \eqref{assurdo2}.
\\Case II) Assume that there exists $r_1\in [0,+\infty)$ such that $$e^{-K\bar{\tau}}\min_{\sigma\in [0,r_1]}\psi(\sigma)<\frac{e^{-2K\bar{\tau}}}{\bar{\tau}}.$$
Note that, for all $r\geq r_1$, it holds that
$$\min_{\sigma\in [0,r]}\psi(\sigma)\leq \min_{\sigma\in [0,r_1]}\psi(\sigma),$$
from which 
$$e^{-K\bar{\tau}}\min_{\sigma\in [0,r]}\psi(\sigma)< \frac{e^{-2K\bar{\tau}}}{\bar{\tau}},\quad \forall r\geq r_1.$$
Thus, using \eqref{infint} we can write $$\int_{0}^{+\infty}\min\left\{e^{-K\bar{\tau}}\min_{\sigma\in [0,r]}\psi(\sigma),\frac{e^{-2K\bar{\tau}}}{\bar{\tau}}\right\}dr\geq \int_{r_1}^{+\infty}\min\left\{e^{-K\bar{\tau}}\min_{\sigma\in [0,r]}\psi(\sigma),\frac{e^{-2K\bar{\tau}}}{\bar{\tau}}\right\}dr$$$$= e^{-K\bar{\tau}}\int_{r_1}^{+\infty}\min_{\sigma\in [0,r]}\psi(\sigma)dr=+\infty.$$
Hence, also in this case we get a contradiction.
\\As a consequence, in all the two possible situations we get a contradiction and we deduce the existence of a positive constant $d^*$ for which inequality \eqref{firstcond} is fulfilled.
\\Finally, we define
	$$\phi^{*}:=\min\left\{e^{-K\bar{\tau}}\psi_{*},\frac{e^{-2K\bar{\tau}}}{\bar{\tau}}\right\},$$
	where $$\psi_{*}=\min_{r\in[0,d^{*}]}\psi(r).$$
	Note that $\phi^{*}>0$, being $\psi$ a positive function. Also, from \eqref{firstcond}, it comes that $$\psi_{*}\leq \min\left\{\psi(r):r\in \left[0,\bar{\tau}R_{V}^{0}+2M_{X}^{0}+\max_{s\in [-\bar{\tau},t]}d_{X}(s)\right]\right\}=\psi_{t},$$
	for all $t\geq -\bar{\tau}$. Thus, we get $$\phi^{*}\leq \phi(t), \quad \forall t\geq -\bar{\tau}.$$
	This implies that, for each $n\geq 2$
	\begin{equation}\label{phi*}
		\begin{split}
		\left(1-e^{-K\bar{\tau}}\int_{n\bar{\tau}}^{n\bar{\tau}+\bar{\tau}}\phi(s)ds\right)^{\frac{1}{3}}&\leq\left(1-e^{-K\bar{\tau}}\int_{n\bar{\tau}}^{n\bar{\tau}+\bar{\tau}}\phi^{*} ds\right)^{\frac{1}{3}}\\&\vspace{0.5cm}=\left(1-e^{-K\bar{\tau}}\phi^{*}\bar{\tau}\right)^{\frac{1}{3}} ,
		\end{split}
	\end{equation}
	with $\left(1-e^{-K\bar{\tau}}\phi^{*}\bar{\tau}\right)^{\frac{1}{3}}<1$.
	\\Next, we set $$C=\frac{1}{3\bar{\tau}}\ln\left(\frac{1}{1-e^{-K\bar{\tau}}\phi^{*}\bar{\tau}}\right)>0.$$
	Notice that $C$ is a constant independent of $N$. Moreover, we have $$\left(1-e^{-K\bar{\tau}}\phi^{*}\bar{\tau}\right)^{\frac{1}{3}}=e^{-C\bar{\tau}},$$
	so that \eqref{phi*} becomes
	\begin{equation}\label{C}
		\left(1-e^{-K\bar{\tau}}\int_{n\bar{\tau}}^{n\bar{\tau}+\bar{\tau}}\phi(s)ds\right)^{\frac{1}{3}}\leq e^{-C\bar{\tau}}, \quad \forall n\geq 2.
	\end{equation} Now we claim that, for each $n\geq 2$, it holds
	\begin{equation}\label{seccond}
		\mathcal{D}(n\bar{\tau})\leq D_{0}e^{-C(n-2)\bar{\tau}}.
	\end{equation}
	Indeed, by induction, if $n=2$ then trivially $\mathcal{D}(2\bar{\tau})=D_{0}$ and the claim holds. So suppose \eqref{seccond} holds true for $n>2$ and prove it for $n+1$. From the induction hypothesis and by recalling of \eqref{C}, we can write $$\mathcal{D}(n\bar{\tau}+\bar{\tau})=\mathcal{D}(n\bar{\tau})\left(1-e^{-K\bar{\tau}}\int_{n\bar{\tau}}^{n\bar{\tau}+\bar{\tau}}\phi(s)ds\right)^{\frac{1}{3}}$$$$\leq D_{0}e^{-C(n-2)\bar{\tau}}e^{-C\bar{\tau}}=D_{0}e^{-C(n+1-2)\bar{\tau}}.$$
	Hence, from \eqref{boundIn} and \eqref{seccond} it follows that, for each $t>2\bar{\tau}$, if $t\in (n\bar{\tau},n\bar{\tau}+\bar{\tau})$, for some $n\geq2$, $$d_{V}(t)\leq D_{n+1}\leq \mathcal{D}(n\bar{\tau}+\bar{\tau})\leq D_{0}e^{-C(n+1-2)\bar{\tau}}\leq D_{0}e^{-C(t-2\bar{\tau})}.$$
	Thus, combining this with the fact that, for all $[-\bar{\tau},2\bar{\tau}]$,$$d_{V}(t)\leq D_{0}\leq D_{0}e^{-C(t-2\bar{\tau})},$$
	we can conclude that estimate \eqref{vel} holds too.
\end{proof}

\bigskip

\parindent=0pt
\providecommand{\Dec}[1]{\textbf{Declarations:} #1}
\Dec
\bigskip

\parindent=0pt
{\bf Ethics approval and consent to participate}: Not applicable.

{\bf Consent for publication}: Not  applicable.

{\bf Availability of data and materials}: Not applicable. 

{\bf Competing interests}:  The author declares that she has no competing interests.

{\bf Funding}: Not applicable.

{\bf Authors' contributions}: The author confirms sole responsibility for  conceptualization, investigation, analysis, writing and editing.

{\bf Acknowledgements}: The author 
acknowledges the support of  UNIVAQ.

\end{document}